\documentclass[]{interact}

\usepackage{epstopdf}
\usepackage[caption=false]{subfig}

\usepackage[numbers,sort&compress]{natbib}
\bibpunct[, ]{[}{]}{,}{n}{,}{,}
\renewcommand\bibfont{\fontsize{10}{12}\selectfont}
\makeatletter
\def\NAT@def@citea{\def\@citea{\NAT@separator}}
\makeatother

\usepackage[latin1]{inputenc}

\usepackage{xcolor}

\allowdisplaybreaks[3]

\theoremstyle{plain}
\newtheorem{theorem}{Theorem}[section]
\newtheorem{lemma}[theorem]{Lemma}

\newtheorem{corollary}[theorem]{Corollary}

\theoremstyle{definition}

\theoremstyle{remark}
\newtheorem{remark}[theorem]{Remark}
\newtheorem*{remark*}{Remark}

\numberwithin{equation}{section}

\newcommand\D{{\mathcal D}}
\newcommand\A{{\mathcal A}}

\newcommand\RR{{\mathbb R}}
\newcommand\ZZ{{\mathbb Z}}

\newcommand\U{{\mathfrak{U}}}

\newcommand\Rc{{\mathcal R}}

\newcommand\R{{\mathfrak R}}

\newcommand\alp{\tilde{\mathfrak D}}
\newcommand\al{\mathfrak D}

\newcommand\Sh{\mbox{\Large $\mathfrak {s}$}}

\catcode`,\active

\catcode`\,12

\begin{document}

\title{Bispectral Laguerre type polynomials}

\author{
\name{Antonio J. Dur\'an\textsuperscript{a} and Manuel D. de la Iglesia\textsuperscript{b}\textsuperscript{*}\footnote{*Corresponding author. Email: mdi29@im.unam.mx}}
\affil{\textsuperscript{a} Departamento de An\'{a}lisis Matem\'{a}tico. Universidad de Sevilla. Apdo (P. O. BOX) 1160. 41080 Sevilla. Spain.
\textsuperscript{b}Instituto de Matem\'aticas. Universidad Nacional Aut\'onoma de M\'exico. Circuito Exterior, C.U. 04510, Ciudad de M\'exico. M\'exico.}
}

\maketitle

\begin{abstract}
We study the bispectrality of Laguerre type polynomials, which are defined by taking suitable linear combinations of a fixed number of consecutive Laguerre polynomials. These Laguerre type polynomials are eigenfunctions of higher-order differential operators and include, as particular cases, the Krall-Laguerre polynomials. As the main results, we prove that these Laguerre type polynomials always satisfy higher-order recurrence relations (i.e., they are bispectral). We also prove that the Krall-Laguerre families are the only polynomials which are orthogonal with respect to a measure on the real line.
\end{abstract}

\begin{keywords}
Orthogonal polynomials; Bispectral orthogonal polynomials; Recurrence relations; Krall polynomials; Laguerre polynomials.
\end{keywords}

\begin{amscode}
42C05; 33C45; 33E30.
\end{amscode}

\section{Introduction and results}
The explicit solution of certain mathematical models of physical interest (in Statistical Mechanics, Potential Theory, Electromagnetism or Quantum Mechanics) often depends on the use of certain special functions which in many cases turn out to be classical families of orthogonal polynomials. These families can be seen as solutions of a special type of bispectral problems.
Bispectrality in its continuous-continuous version is a subject that was started by H. Duistermaat and F.A. Grünbaum in the 1980s \cite{DuiGr}.  One special case of the bispectral problem is the problem of determining \textit{classical orthogonal polynomials}, a notion that goes back to S. Bochner \cite{Boc} (see also \cite{Rou}) who proved in 1929 that the families of Hermite, Laguerre and Jacobi (and Bessel, if non-positive measures are considered) are the only families of polynomials $(q_n(x))_n$, $\deg q_n=n$, which are orthogonal with respect to a measure on the real line and, in addition, are eigenfunctions of a second-order differential operator (acting on the continuous variable $x$).

The orthogonality of a sequence of polynomials  with respect to a measure on the real line is equivalent to a second-order spectral problem in the discrete variable. This result is known as Favard's Theorem and establishes that a sequence $(q_n)_n$ of polynomials, $\deg (q_n)=n$, is orthogonal with respect to a measure (not necessarily positive) on the real line if and only if it satisfies a so-called three-term recurrence relation of the form
\begin{equation}\label{ttrr}
xq_n(x)=a_nq_{n+1}(x)+b_nq_n(x)+c_nq_{n-1}(x), \quad n\ge 0,
\end{equation}
where $(a_n)_n$, $(b_n)_n$ and $(c_n)_n$ are sequences of real numbers with $c_n\not =0$, $n\ge 1$. If we write $D_n$ for the second-order difference operator (acting on the discrete variable $n$)
\begin{equation*}\label{ottrr}
D_n=a_n\Sh_1+b_n\Sh_0+c_n\Sh_{-1},
\end{equation*}
where $\Sh_l$ stands for the shift operator $\Sh_l(f(n))=f(n+l)$, the three-term recurrence relation (\ref{ttrr}) can be expressed in the spectral form as $D_n(q_n)=xq_n$.

Following the Duistermaat-Grünbaum terminology, one can say that a sequence of polynomials $(q_n(x))_n$ is bispectral if there exist a difference operator acting on the discrete variable $n$ of the form
\begin{equation}\label{doho2g}
D_n=\sum_{i=s}^r\gamma_{n,i}\Sh _i, \quad s\le r, \quad s,r\in \ZZ,
\end{equation}
where $\gamma_{n,i}$, $i=s,\ldots, r$, are sequences of numbers with $\gamma_{n,s},\gamma_{n,r}\not =0 $, $n\ge 0$, and an operator acting on the continuous variable $x$, with respect to which the polynomials $(q_n(x))_n$ are eigenfunctions. In this paper we only consider differential operators acting on the continuous variable $x$.

It is easy to see that if $D_n(q_n)=Q(x)q_n$ then $Q$ is a polynomial of degree $r$, and hence each operator $D_n$ of the form (\ref{doho2g}) produces a higher-order recurrence relation for the polynomials $(q_n)_n$, i.e.
\begin{equation}\label{qrr}
Q(x)q_n(x)=\sum_{i=s}^r\gamma_{n,i}q_{n+i}(x),\quad s\le r.
\end{equation}

Besides the classical polynomials, Krall polynomials are other well-known examples of bispectral polynomials. Krall polynomials are  eigenfunctions of higher-order differential operators. They are called Krall polynomials because they were introduced by H.L. Krall in 1940 \cite{Kr2}: Krall proved that the differential operators must have even order and classified the case of order four. Since the 1980's, Krall polynomials associated with differential operators of any even order have been constructed and intensively studied (\cite{koekoe,koe,koekoe2,L1,L2,GrH1,GrHH,GrY,Plamen1,Plamen2,Zh}; the list is not exhaustive).

Krall-Laguerre polynomials are orthogonal with respect to measures of the form
\begin{equation}\label{Kcwm}
\mu_{\alpha}^{\{b_h\}}=x^{\alpha-m}e^{-x}+\sum_{h=0}^{m-1}b_h\delta_0^{(h)},\quad x\ge 0,
\end{equation}
where $\alpha $ and $m$ are positive integers with $\alpha\geq m$ and $b_h$, $h=0,\ldots, m-1$, are certain real numbers with $b_{m-1}\not=0$. Krall-Laguerre polynomials are also eigenfunctions of a higher-order differential operator.

Other examples of bispectral polynomials are the Krall-Jacobi polynomials (see \cite{Kr2,GrH1,GrY,Plamen1,koe,koekoe2}), the Krall-Sobolev polynomials (see \cite{KKB,Ba,ddI1,ddI3}), the exceptional polynomials (see \cite{GUKM1,duch,dume,durr,duhj,GFGM}, and references therein) or the Grünbaum and Haine extension of Krall polynomials (\cite{GrH3}; see also \cite{Plamen1,du1}). In these cases, the associated operators (in the discrete and continuous variable) have order bigger than $2$.

\medskip

From the Laguerre polynomials $(L_n^\alpha)_n$, we can generate sequences of polynomials $(q_n(x))_n$ which are eigenfunctions of a higher-order differential operator (acting on the continuous variable $x$) in the following way. Consider a finite set $G=\{g_1,\ldots, g_m\}$ of positive integers (written in increasing size) and polynomials $\Rc_g$, $g\in G$, with degree of $\Rc_g$ equal to $g$. We associate to them the Casoratian determinant
\begin{equation}\label{cdt}
\Omega_G (x)=\det \left(\Rc_{g_l}(x-j)\right)_{l,j=1}^m,
\end{equation}
and assume that
\begin{equation}\label{mai}
\Omega_G(n)\not =0, \quad n=0,1,2,\ldots
\end{equation}
We then define the sequence of polynomials $(q_n)_n$ by
\begin{equation}\label{quschi}
q_n(x)=\begin{vmatrix}
L^\alpha_n(x) & L^\alpha_{n-1}(x) & \ldots & L^\alpha_{n-m}(x) \\
\Rc_{g_1}(n) & \Rc_{g_1}(n-1) & \ldots &
\Rc_{g_1}(n-m) \\
               \vdots & \vdots & \ddots & \vdots \\
\Rc_{g_m}(n) & \displaystyle
\Rc_{g_m}(n-1) & \ldots &\Rc_{g_m}(n-m)
             \end{vmatrix}.
\end{equation}
The assumption (\ref{mai}) says that the determinant on the right-hand side of (\ref{quschi}) defines a polynomial of degree $n$, $n\ge 0$. Expanding the determinant by its first row, we see that each $q_n$, $n\ge m$, is a linear combination of $m$ consecutive Laguerre polynomials.

Using the $\D$-operator method, it is proved in \cite{ddI1} (see Lemma 3.1 and Theorem 3.2) that the polynomials $(q_n)_n$ are eigenfunctions of a higher-order differential operator (acting on the continuous variable $x$) of the form
$D_x=\sum_{l=0}^rh_l(x)\left(\frac{d}{dx}\right)^l$,
where $h_l(x)$ are polynomials and $r$ is a positive even integer greater than $2$. This differential operator can, in fact, be explicitly constructed. For a different approach of the polynomials (\ref{quschi}) using discrete Darboux transformations see \cite{GrHH,Plamen2}.

The most interesting case is when $\alpha$ is a positive integer with $\alpha \ge m$, $G=\{\alpha,\alpha+1,\ldots, \alpha+m-1\}$ and for $h=1,\ldots, m$,
\begin{equation}\label{rpm}
\R_{g_h}(x)=\binom{x+\alpha+h-1}{\alpha+h-1}+(h-1)!\sum_{l=0}^{h-1}\frac{(-1)^la_{h-l-1}}{(\alpha-l)_l}\binom{x+l}{l},
\end{equation}
where $a_l$, $l=0,\ldots, m-1$, are real numbers with $a_0\not=0$.
Then the polynomials $(q_n)_n$ are orthogonal with respect to the Krall-Laguerre measure $\mu_\alpha^{\{b_h\}}$ (\ref{Kcwm}) (for certain choice of the parameters $b_0,\ldots, b_{m-1}$).

\medskip

As the main results of this paper, we first prove that for any set of polynomials $\Rc_g$, $g\in G$, with $\deg \Rc_g=g$, satisfying (\ref{mai}), the polynomials $(q_n)_n$ (\ref{quschi}) are bispectral. And second, we also prove that the polynomials (\ref{rpm}) are essentially the only ones for which the sequence $(q_n)_n$ satisfies a three-term recurrence relation (and therefore they are orthogonal with respect to a measure). 

\medskip

The content of the paper is as follows. After some preliminaries in Section 2, in Section 3, we find some orthogonality properties for the polynomials $(q_n)_n$, when $\alpha-\max G\not =0,-1,-2,\ldots $. For that we will modify the Laguerre weight   with a nonsymmetric perturbation (which strongly depends on the polynomials $\Rc_g(x)$). These orthogonality properties allow us to prove in Section 4 that the sequence $(q_n)_n$ satisfies some recurrence relations of the form (\ref{qrr}) where $s=-r$ (we also prove that (\ref{qrr}) always holds with $Q(x)=x^{\max G+1}$, although in some particular cases, a polynomial of lower degree can be found). On the other hand, the orthogonality properties constrain the number of terms of these recurrence relations: in particular, we prove in Section 4 that the sequence $(q_n)_n$ can never satisfy a three-term recurrence relation of the form (\ref{ttrr}).  We also prove some results for the algebra of operators $\al_n$. This algebra is defined as follows. We denote by $\A_n$ the algebra formed by all higher-order difference operators (acting on the variable $n$) of the form (\ref{doho2g}). Then we define
\begin{equation}\label{algo}
\al _n=\{D_n\in \A_n: D_n(q_n)=Q(x)q_n,\; Q\in\RR[x]\},
\end{equation}
where $\RR[x]$ denotes the linear space of real polynomials in the unknown $x$. This algebra is actually characterized by the algebra of polynomials defined from the corresponding eigenvalues
\begin{equation}\label{algp}
\alp _n=\{Q\in\RR[x]: \mbox{there exists $D_n\in \al_n$ such that $D_n(q_n)=Q(x)q_n$}\}.
\end{equation}
In Section 4 we prove that when $G$ is a segment, i.e. its elements are consecutive positive integers, the algebra $\alp_n$ has a simple estructure:
$$
\alp_n=\{p(x)x^{\max G+1}+c:p\in \RR[x],c\in \RR\}.
$$
For a characterization of the corresponding algebra for the Charlier-type polynomials see \cite{ducb}. We also give some examples showing that, in general, this algebra can have a more complicated structure (see \cite{Plamen2} for a characterization of the algebra of differential operators acting on the variable $x$ associated to the Krall-Laguerre polynomials orthogonal with respect to the measure (\ref{Kcwm})).

In Section 5 we study the case when $\alpha=1,\ldots, \max G$ (which includes the Krall-Laguerre polynomials orthogonal with respect to (\ref{Kcwm})). In order to get orthogonality properties, we have to transform a portion of the bilinear form studied in Section 4 in a discrete Sobolev part. As the main result here, we prove that the sequence $(q_n)_n$ satisfies a three-term recurrence relation of the form (\ref{ttrr}) only when $\{\R_g(x)\}_{g\in G}$ has the form (\ref{rpm}) and hence, this is the only possible choice such that the polynomials $(q_n)_n$ are orthogonal with respect to a measure.

The case of the Jacobi type polynomials is more involved since they are defined from two sets of positive integers $G$ and $H$ and two families of polynomials $\Rc_g$, $g\in G$, $\mathcal S_h$, $h\in H$. We guess that they should also be bispectral but the proof remains as a challenge.

\section{Preliminaries}
For $\alpha\in\mathbb{R}$, we use the standard definition of the Laguerre polynomials $(L_n^{\alpha})_n$:
\begin{equation*}\label{deflap}
L_n^{\alpha}(x)=\sum_{j=0}^n\frac{(-x)^j}{j!}\binom{n+\alpha}{n-j}.
\end{equation*}
For $\alpha\neq-1,-2,\ldots$, they are orthogonal with respect to a measure which we denote by $\mu_{\alpha}=\mu_{\alpha}(x)dx$ (normalized so that
$\int_0^\infty d\mu_{\alpha}(x)=\Gamma(\alpha+1))$. This measure is positive only when $\alpha>-1$ and then $\mu_{\alpha}=x^\alpha e^{-x}$.

We consider a finite set $G=\{g_1,\ldots, g_m\}$ of $m$ positive integers (written in increasing size: $g_i<g_j$) and polynomials $\Rc_g$, $g\in G$, with $\deg \Rc_g=g$. Since the leading coefficients of the polynomials $\Rc_g$, $g\in G$, just produce a renormalization of the polynomials $(q_n)_n$ (\ref{quschi}), we assume along this paper that
\begin{equation*}\label{lcqn}
\Rc_g(x)=\frac{1}{g!}x^g+\mbox{terms of lower degree}.
\end{equation*}
We assume that the Casoratian determinant $\Omega_G$ (\ref{cdt}) satisfies $\Omega_G (n)\not =0$, $n\ge 0$.

\begin{lemma}\label{lepm} There exist numbers $\kappa _i^g$, $i=0,\ldots, m-1$, $g\in G$, such that for $i=0,\ldots, m-1,$ we have
\begin{align}\label{epm1}
\sum_{g\in G}\kappa_i^g\Rc_g(-j)&=0,\quad j=1,\ldots, m-1-i,\\\label{epm2}
\sum_{g\in G}\kappa_i^g\Rc_g(-m+i)&\not =0.
\end{align}
\end{lemma}

\begin{proof}
For each $i$, we can see (\ref{epm1}) and (\ref{epm2}) as a system of $m-i$ linear equations in the $m$ unknowns $\kappa_i^g$, $g\in G$. The rows of the $(m-i)\times m$ coefficient matrix coincide with the $m-i$ first columns of the Casoratian matrix $(\Rc_g(-j))_{\begin{subarray}{1} g\in G,\\
j=1,\ldots, m\end{subarray}}$. The determinant of this Casoratian matrix is $\Omega_G(0)\not =0$ (\ref{cdt}). Therefore the coefficient matrix has full rank $m-i$, and the linear system of equations has always a solution.
\end{proof}

We associate to $G$ and $\{\Rc_g\}_{g\in G}$ the sequence of polynomials $(q_n)_n$ defined by (\ref{quschi}).

\begin{remark}\label{r2.1}
We stress that if we substitute the polynomials $\Rc_g$ in the determinant (\ref{quschi}) by any linear combination $R_g$ of the form
\begin{equation}\label{ocl}
R_g=\Rc_g+\sum_{\tilde g\in G;\tilde g<g}\zeta_{g,\tilde g}\Rc_{\tilde g}
\end{equation}
the polynomials $q_n$, $n\ge 0$, remain invariant. Notice that $\deg R_g=\deg \Rc_g=g$, and the leading coefficient of $R_g$ is again $1/g!$.
\end{remark}
We set $\R_g$, $g\in G$, for the (unique) linear combination of the form (\ref{ocl}) such that
\begin{align}\label{aqj}
&\mbox{the powers $x^{\tilde g}$, $\tilde g<g, \tilde g\in G$, do not appear}\\\nonumber
&\mbox{in the expansion of $\R_g$ in powers of $x$}.
\end{align}
The set of polynomials $\{\R_g\}_{g\in G}$ is called the \emph{reduced representation} of $(q_n)_n$.

\medskip
As we wrote in the Introduction, the most interesting case is when $\alpha$ is a positive integer with $\alpha\ge m$, $G=\{\alpha,\alpha+1,\ldots, \alpha+m-1\}$ and $\R_{g_h}$, $h=1,\ldots, m$, defined by (\ref{rpm}). The polynomials $(q_n)_n$ are then orthogonal with respect to the measure (\ref{Kcwm}). In \cite[(1.3) and Example 1, p. 86]{ddI1}, we represent $(q_n)_n$ with a different set of polynomials $\{R_g\}_{g\in G}$, from where the reduced representation $\{\R_g\}_{g\in G}$ (\ref{rpm}) can be easily obtained. If
$1\le \alpha\le m-1$ and if we take
\begin{equation}\label{rpm1}
\R_{g_h}(x)=\binom{x+\alpha+h-1}{\alpha+h-1}+(h-1)!\sum_{l=0}^{h+\alpha-m-1}\frac{(-1)^l\tilde a_{h-l-1}}{(\alpha-l)_l}\binom{x+l}{l},
\end{equation}
then the polynomials $(q_n)_n$ satisfies also three-term recurrence relations (along this paper, we always take $\sum_{l=u}^{v}\rho_l=0$ for $u>v$). But these polynomials $(q_n)_n$ are somehow degenerated. Indeed, since \begin{equation}\label{dlp}
(L_n^\alpha )^{(j)}(0)=(-1)^j\binom{n+\alpha}{\alpha + j},
\end{equation}
it is not difficult to see that
$q_n^{(j)}(0)=0$, $j=0,\ldots , m-\alpha-1$, $n\ge m-\alpha$. Moreover, the polynomials $(q_{n+m-\alpha}(x)/x^{m-\alpha})_{n\ge m-\alpha}$ are particular examples of the case (\ref{rpm}) replacing $\alpha$ by $m$, $m$ by $\alpha$ and writing
$a_j=(-1)^{\alpha-m}\frac{(\alpha-1)!}{(m-1)!}\tilde a_{\alpha-m+j}$, $j=0,\ldots, \alpha-1$ (notice that only the parameters $\tilde a_j$, $j=m-\alpha,\ldots , m-1$, appear in (\ref{rpm1})).

\medskip

We will use the following alternative definition of the polynomials $(q_n)_n$ (\ref{quschi}).
For $j=0,\ldots, m$, and $\{R_g\}_{g\in G}$ as in (\ref{ocl}), the sequence $\beta _{n,j}$ is defined by
\begin{equation}\label{defb}
\beta_{n,j}=(-1)^j\det (R_g(n-i))_{\begin{subarray}{1} g\in G,\\
i=0,\ldots, m,i\not =j\end{subarray}}.
\end{equation}
Notice that
\begin{equation*}\label{nnc}
\beta_{n,m}=(-1)^m\Omega_G(n+1)\not =0,
\end{equation*}
as a consequence of (\ref{mai}). By expanding the determinant (\ref{quschi}) by its first row (writing $L_{u}^\alpha(x)=0$ for $u<0$), we get for the polynomial $q_n$ the expansion
\begin{equation*}\label{adefqn}
q_n(x)=\sum_{j=0}^{m\wedge n}\beta_{n,j}L_{n-j}^\alpha(x).
\end{equation*}
On the other hand, for $g\in G$, we have
$$
\det\begin{pmatrix}R_g(n)&\ldots &R_g(n-m)\\
R_{g_1}(n)&\ldots &R_{g_1}(n-m)\\
\vdots&\ddots&\vdots\\
R_{g_m}(n)&\ldots &R_{g_m}(n-m)\end{pmatrix}=0,
$$
because the matrix has two equal rows. Expanding it by its first row, we get
\begin{equation}\label{defbetas}
\sum_{j=0}^{m}\beta_{n,j}R_g(n-j)=0,\quad g\in G.
\end{equation}
We point out that the coefficients $\beta_{n,j}$ are invariant under a substitution of the polynomials $\Rc_g$ in the determinant (\ref{quschi}) by any linear combination of the form (\ref{ocl}).

\medskip

Given polynomials $p_i$, $i=1,\ldots,s$, we define the Casoratian determinant
\begin{equation}\label{defcd}
\Phi_s(x)\equiv\Phi ^{p_1,\ldots,p_s}(x)=\det(p_i(x-j))_{\begin{subarray}{1} i=1,\ldots, s,\\
j=0,\ldots, s-1\end{subarray}}.
\end{equation}
It is not difficult to see that
\begin{equation}\label{gcd}
\deg \Phi_s(x)\begin{cases}\displaystyle=\sum_{i=1}^s\deg p_i-\binom{s}{2}, & \deg p_i\not =\deg p_j, i\not =j,\\
\displaystyle<\sum_{i=1}^s\deg p_i-\binom{s}{2}, &\mbox{otherwise}.\end{cases}
\end{equation}

We will also need the following combinatorial formula: if $\alpha ,k,l$ are nonnegative integers with $l\geq\alpha+k$, then
\begin{equation}\label{fcc}
\sum_{j=0}^{l-\alpha-k}(-1)^{j}\binom{\alpha+j+k-l-1}{j}\binom{\alpha+u}{\alpha+j}=
\binom{u+l-k}{l-k}.
\end{equation}

\section{Orthogonality properties}
In this section we establish some orthogonality properties for the polynomials $(q_n)_n$, from which we will deduce in the next section the recurrence relations for them.
We associate to the polynomials $\Rc_g$, $g\in G$, the rational functions $U_g^i$, $i=0,\ldots, m-1$, $g\in G$, defined as follows.
Given a  polynomial $\Rc_g$ of degree $g$ and leading coefficient equal to $1/g!$, we can always write
\begin{equation}\label{defr}
\Rc_g(x)=\sum_{l=0}^{g} w_l^g\binom{x+l}{l},
\end{equation}
for certain numbers $w_l^g$, $l=0,\ldots , g$, with $w_g^g=1$ (which are uniquely determined from the polynomial $\Rc_g$).
For each $i=0,\ldots, m-1$, we define the rational functions $U_i^g$ and $\U_i$ as
\begin{align}
\nonumber U_i^g(x)&=\frac{-x^{-m+i}}{m}+\kappa_i^g\sum_{l=0}^g(\alpha-l)_l w_l^g x^{-l-1},\\
\label{defu2}\U_i(x)&=\sum_{g\in G}U_i^g(x),
\end{align}
where the numbers $\kappa_i^g$, $g\in G$, satisfy (\ref{epm1}) and (\ref{epm2}), and $(a)_n=a(a+1)\ldots(a+n-1)$ denotes the Pochhammer symbol.

Taking into account that $\Rc_g(-j)=\sum_{l=0}^{j-1}w_l^g\binom{-j+l}{l}$, $j=1,\ldots, m-1$, (\ref{epm1}), (\ref{epm2}) and after some computations, we deduce the following alternative representation of $\U_i$
\begin{equation}\label{defu3}
\U_i(x)=-x^{-m+i}+\sum_{l=m-i-1}^{\max G}(\alpha-l)_lx^{-l-1}\sum_{g\in G;g\ge l}\kappa_i^gw_l^g.
\end{equation}

For  $\alpha\in \RR$ with $\alpha-\max G\not =0,-1,-2,\ldots $, we consider the following bilinear form (which is nonsymmetric in general):
\begin{equation}\label{innlag}
\langle p,q\rangle=\langle p,q\rangle_1+\langle p,q\rangle_2,
\end{equation}
where
\begin{align}\label{ip1}
\langle p,q\rangle_1&=\int_0^\infty p(x)q(x)\mu_{\alpha-m}dx,
\end{align}
with $\mu_{\alpha-m}$ denoting the Laguerre weight (with parameter $\alpha-m$) and
\begin{align}\label{innlag2}
\langle p,q\rangle_2&=\sum_{i=0}^{m-1}\frac{q^{(i)}(0)}{i!}\int_0^\infty p(x)\U_i(x)\mu_{\alpha}dx,
\end{align}
where $\U_i$ is defined by \eqref{defu2}. The following lemma will be the key for most of our results.

\begin{lemma}\label{lsc} For $k,u\ge 0$, with $u\ge k$, and $0\le i\le m-1$, we have
\begin{equation*}\label{fsc}
\langle x^kL_u^\alpha (x),x^i\rangle =\Gamma(\alpha)\sum_{g\in G}\kappa _i^g \sum _{l=k}^g(\alpha-l)_kw_l^g\binom{u+l-k}{l-k}.
\end{equation*}
\end{lemma}

\begin{proof}
Using the expansion $L_{n}^{\alpha}=\sum_{j=0}^n\frac{(\alpha-\beta)_j}{j!}L_{n-j}^{\beta}$
(see \cite[p. 192, (39)]{EMOT}),  we get
$$
\int_0^\infty L_{n}^{\alpha}\mu_{\alpha -l}(x)dx=\Gamma(\alpha-l+1)\binom{n+l-1}{l-1},
$$
where, as usual, for a negative integer $b$ and a real number $a>0$, we take $\binom{a}{b}=0$.
The lemma follows now from the following computations:
\begin{align*}
\langle x^kL_u^\alpha (x),x^i\rangle _2&=\int_0^\infty x^kL_u^\alpha (x)\U_i(x)\mu_{\alpha}dx\\\nonumber
&=-\int_0^\infty x^kL_u^\alpha (x)\mu_{\alpha-m+i}dx
+\sum_{g\in G}\kappa_i^g\sum_{l=0}^g(\alpha-l)_lw_l^g\int_0^\infty x^kL_u^\alpha (x)\mu_{\alpha-l-1}dx\\\nonumber
&=-\int_0^\infty L_u^\alpha (x)\mu_{\alpha-m+k+i}dx
+\sum_{g\in G}\kappa_i^g\sum_{l=0}^g(\alpha-l)_lw_l^g\int_0^\infty L_u^\alpha (x)\mu_{\alpha+k-l-1}dx\\\nonumber
&=-\langle x^kL_u^\alpha (x),x^i\rangle _1+\sum_{g\in G}\kappa_i^g\sum_{l=0}^g(\alpha-l)_lw_l^g\Gamma(\alpha+k-l)\binom{u+l-k}{l-k}\\\nonumber
&=-\langle x^kL_u^\alpha (x),x^i\rangle _1+\Gamma(\alpha)\sum_{g\in G}\kappa_i^g\sum _{l=k}^g(\alpha-l)_kw_l^g\binom{u+l-k}{l-k}.
\end{align*}
Observe that the result holds without any assumptions on the numbers $\kappa_i^g$.
\end{proof}

We are now ready to prove the main result in this section.

\begin{theorem} \label{occ} Let $\alpha$ be a real number with $\alpha-\max G\not =0,-1,-2,\ldots $. Assume
that conditions (\ref{mai}) hold and take numbers $\kappa _i^g$, $i=0,\ldots, m-1$, $g\in G$, as in Lemma \ref{lepm}. Then, for $n\geq0$, the polynomials $(q_n)_n$ satisfy the following orthogonality properties with respect to the bilinear form (\ref{innlag}):
\begin{align*}
\langle q_n,q_i\rangle &=0, \quad i=0,\ldots ,n-1,\\
\langle q_n,q_n\rangle &\not=0.
\end{align*}
\end{theorem}

\begin{proof}
Assume first that $i\ge m$. Then
$$
\langle q_n,x^i\rangle =\langle q_n,x^i\rangle _1=\int_0^\infty q_n(x)\mu_{\alpha-m+i}dx
\begin{cases}=0, &m\le i\le n-1,\\
\not =0, &m\le i=n.
\end{cases}
$$
Let us now assume $i\le m-1$. Lemma \ref{lsc} for $k=0$ and (\ref{defr}) gives
$$
\langle L_u^\alpha (x),x^i\rangle =\Gamma(\alpha)\sum_{g\in G}\kappa_i^g\Rc_g(u).
$$
Hence for $n\ge m$ and using \eqref{defbetas} we have
\begin{align*}
\langle q_n(x),x^i\rangle &=\sum_{j=0}^m\beta_{n,j}\langle L_{n-j}^\alpha (x),x^i\rangle =
\sum_{j=0}^m\beta_{n,j}\Gamma(\alpha)\sum_{g\in G}\kappa_i^g\Rc_g(n-j)\\
&=\Gamma(\alpha)\sum_{g\in G}\kappa_i^g\sum_{j=0}^m\beta_{n,j}\Rc_g(n-j)=0.
\end{align*}
Notice that at this point we have not used yet any assumptions on the numbers $\kappa_i^g$.

For $0\le n\le m-1$, we have
\begin{align*}
\langle q_n(x),x^i\rangle &=\Gamma(\alpha)\sum_{g\in G}\kappa_i^g\sum_{j=0}^n\beta_{n,j}\Rc_g(n-j)\\
&=-\Gamma(\alpha)\sum_{g\in G}\kappa_i^g\sum_{j=n+1}^m\beta_{n,j}\Rc_g(n-j)\\
&=-\Gamma(\alpha)\sum_{j=1}^{m-n}\beta_{n,n+j}\sum_{g\in G}\kappa_i^g\Rc_g(-j).
\end{align*}
Hence for $0\le i\le n-1$, it follows from (\ref{epm1}) that $\langle q_n(x),x^i\rangle=0$. If $i=n$,
(\ref{epm2}) gives
$$
\langle q_n(x),x^n\rangle=-\Gamma(\alpha)\beta_{n,m}\sum_{g\in G}\kappa_n^g\Rc_g(-m+n)\not=0.
$$
\end{proof}

\begin{corollary}\label{cor1} With no assumptions on the real numbers $\kappa_j^g$, $j=0,\ldots, m-1$, $g\in G$, we always have
$\langle q_n(x),x^i\rangle=0$, $n\ge m, 0\le i\le n-1$.
\end{corollary}

\section{Recurrence relations}\label{lrr}
In this section we prove one of the main results of this paper.

\begin{theorem}\label{t4.1} Let $\alpha$ be a real number with $\alpha-\max G\not =0,-1,-2,\ldots $.
For any $p\in \RR[x]$ the sequence $(q_n)_n$ satisfies the recurrence relation
$$
x^{\max G+1}p(x)q_n(x)=\sum_{j=-s}^s\gamma_{n,j}q_{n+j}(x),\quad \gamma_{n,s}, \gamma_{n,-s}\not =0,
$$
where $s=\deg p+\max G+1$.
\end{theorem}

\begin{proof}
For $u\ge 0$, write $Q(x)=x^{\max G+1}x^u$, and $s_u=u+\max G+1$.

We first prove that
\begin{align}\label{cs1}
\langle Q(x)q_n(x),x^i\rangle &=0, \quad n\ge s_u+1,0\le i\le n-s_u-1,\\\label{cs2}
\langle Q(x)q_n(x),x^{n-s_u}\rangle &\not =0, \quad n\ge s_u+1.
\end{align}
Indeed $x^iQ(x)x^{-m}=x^{s_u+i-m}=r(x),$
which it is always a polynomial of degree at most $n-m-1$. Hence
$$
\langle Q(x)q_n(x),x^i\rangle_1=\int_0^\infty r(x)q_n(x)\mu_{\alpha}dx=0,
$$
because $q_n$ is a linear combination of $L_{n-j}^\alpha$, $j=0,\ldots, n-m$.

Since for $i\ge m$ $\langle Q(x)q_n(x),x^i\rangle_2=0$, we have proved (\ref{cs1}) for
$m\le i\le n-s_u-1$.

Assume now $0\le i\le m-1$ and $i\le n-s_u-1$. Using (\ref{defu3}), we get
\begin{align*}
Q(x)\U_i(x)=-x^{s_u-m+i}+\sum_{l=m-i-1}^{\max G}(\alpha-l)_l x^{s_u-l-1}\sum_{g\in G;g\ge l}k_i^g w_l^g.
\end{align*}
which it is again a polynomial of degree at most $n-m-1$.
Hence
$$
\langle Q(x)q_n(x),x^i\rangle=\int_0^\infty Q(x)\U_i(x)q_n(x)\mu_{\alpha}dx=0,
$$
which proves (\ref{cs1}) also for $0\le i\le m-1$ and $i\le n-s_u-1$.

For $i=n-s_u$, we have as before $q_{n-s_u}(x)Q(x)x^{-m}=r(x)$ and $\deg r=n-m$, so $\langle Q(x)q_n(x),x^{n-s_u}\rangle_1\not =0$. And, as before, $\langle Q(x)q_n(x),x^{n-s_u}\rangle_2 =0$. Hence $\langle Q(x)q_n(x),x^{n-s_u}\rangle \not =0$. From (\ref{cs1}) and (\ref{cs2}) it follows that
$$
\langle Q(x)q_n(x),q_j(x)\rangle\begin{cases}=0, &n\ge s_u-1,0\le j\le n-s_u-1,\\
\not =0, &n\ge s_u-1, j= n-s_u,\end{cases}
$$
from where it is easy to conclude using the orthogonality properties of Theorem \ref{occ}.
\end{proof}

The orthogonality properties proved in the previous section constrain both the number of terms of these recurrence relations and the type of polynomials $Q$ for which a recurrence relation of the form (\ref{qrr}) holds.

\begin{theorem}\label{pzt} Let $\alpha$ be a real number with $\alpha-\max G\not =0,-1,-2,\ldots $, and assume
that conditions (\ref{mai}) hold. Let $Q$ be the polynomial $Q(x)=\sum_{k=u}^v\sigma_kx^k$, with $u\le v$ and $\sigma_u,\sigma_v\not =0$. If there exists $\hat g\in G$ such $\hat g-u\not \in G$ and $\hat g-u\ge 0$ then the polynomials $(q_n)_n$ do not satisfy a recurrence relation of the form (\ref{qrr}).
\end{theorem}

\begin{proof}
We proceed by \textit{reductio ad absurdum}.
Assume that the sequence $(q_n)_n$ satisfies the recurrence relation (\ref{qrr}).
Using Theorem \ref{t4.1} and Corollary \ref{cor1} we get 
$$
\langle Q(x)q_n,1\rangle =\sum_{j=s}^v\gamma_{n,j}\langle q_{n+j},1\rangle=0,
$$ 
for $n\ge m-s$ (with no assumptions on the numbers $\kappa_i^g$). Now, using Lemma \ref{lsc}, we can write
\begin{align*}
0&=\langle Q(x)q_n,1\rangle=\sum_{j=0}^m\beta_{n,j}\langle Q(x)L_{n-j}^\alpha (x),1\rangle
=\Gamma(\alpha)\sum_{j=0}^m\beta_{n,j}\sum_{k=u}^v\sigma_k\langle x^kL_{n-j}^\alpha (x),1\rangle\\
&=\Gamma(\alpha)\sum_{g\in G}\kappa _0^g \sum_{j=0}^m\beta_{n,j}\sum_{k=u}^v\sigma_k\sum _{l=k}^g(\alpha-l)_kw_l^g\binom{n-j+l-k}{l-k}.
\end{align*}
Taking $\kappa_0^g=\delta_{g,\hat g}$ (where $\delta_{i,j}$ denotes the Kronecker's delta), we have
\begin{equation}\label{mv5}
0=\sum_{j=0}^m\beta_{n,j}\sum_{k=u}^v\sigma_k\sum _{l=k}^{\hat g}(\alpha-l)_kw_l^{\hat g}\binom{n-j+l-k}{l-k}.
\end{equation}
Write $r$ for the polynomial
$$
r(x)=\sum_{k=u}^v\sigma_k\sum _{l=k}^{\hat g}(\alpha-l)_kw_l^{\hat g}\binom{x+l-k}{l-k}.
$$
One can now see that the polynomial $r$ has degree $\hat g-u\ge 0$ and leading coefficient equal to $\sigma_u(\alpha-\hat g)_u\not =0$. Consider now the Casoratian determinant defined by the polynomials $r,\mathcal R_g,g\in G$, as in (\ref{defcd}): $\Phi (x)=\Phi ^{r,\{\mathcal R_g\}_{g\in G}}(x)$. From the definition of $\beta_{n,j}$ (\ref{defb}), the identity (\ref{mv5}) can then be rewritten as $0=\Phi (n)$, which is a contradiction because, since $\deg r=\hat g-u\not\in G$, $\Phi$ is a polynomial of
degree $\hat g-u+\sum_{g\in G}g-\binom{m+1}{2}>0$ (see \eqref{defcd}).
\end{proof}

As a consequence of the previous Theorem  we prove that when $\alpha-\max G\not =0,-1,-2,\ldots $,
the sequence $(q_n)_n$ can never satisfy a three-term recurrence relation of the form (\ref{ttrr}).

\begin{corollary}\label{pjc} For $\alpha\in\RR$ with $\alpha-\max G\not =0,-1,-2,\ldots $, assume
that conditions (\ref{mai}) hold. Then the polynomials $(q_n)_n$ never satisfy a three-term recurrence relation as in (\ref{ttrr}) and hence they are not orthogonal with respect to any measure.
\end{corollary}

\begin{proof}

It is enough to apply the previous theorem to $Q(x)=x$ and $\hat g=\min G$.
\end{proof}

Using Theorem \ref{pzt} we can also characterize the algebra of operators $\alp_n$ (\ref{algp}) when $G$ is a segment, i.e. $G=\{f,f+1,\ldots , f+h\}$ for  some integers $f>0$ and $h\geq0$.

It is easy to see that for an operator $D_n\in \al_n$ of the form (\ref{doho2g}), the eigenvalue $Q$ is a polynomial of degree $r$. Moreover, this polynomial $Q$ determines uniquely the operator $D_n$. The map $D_n\to Q(x)$ is an isomorphism between the algebra of operators $\al _n$ and the algebra of polynomials $\alp_n$ (\ref{algp}). So, we study the algebra $\al _n$ by means of the algebra $\alp_n$.

\begin{corollary}\label{c4.4} Let $\alpha$ be a real number with $\alpha-\max G\not =0,-1,-2,\ldots $, and assume
that conditions (\ref{mai}) hold. If $G$ is a segment then
$$
\alp_n=\{p(x)x^{\max G+1}+c:p\in \RR[x],c\in \RR\}.
$$
\end{corollary}

\begin{proof}
Write $G=\{f,f+1,\ldots , f+h\}$. Take now $Q\in \alp_n$ and write $Q(x)-Q(0)=\sum_{k=u}^v\sigma_kx^k$. If we set $\hat g=f+l$, $l=0,\ldots,h-1$, we have that $\hat g-l-1\not \in G$ and $\hat g-l-1\ge 0$. Hence, according to Theorem \ref{pzt}, we have $u\ge h+1$. For $\hat g=f+h$, we have for $l=h+1,\ldots , f+h$ that $\hat g-l\not \in G$ and $\hat g-l\ge 0$, and hence $u\ge f+h+1$. Now it is enough to apply Theorem \ref{t4.1}.
\end{proof}

When $\alpha=1,\ldots, \max G$, Corollary \ref{pjc} is no longer true. Indeed, for $\alpha\ge m$, $G=\{\alpha,\alpha+1,\ldots, \alpha+m-1\}$ and $\R_g$, $g\in G$, as in (\ref{rpm}), the polynomials $(q_n)_n$ are orthogonal with respect to a measure and then $\alp_n=\RR[x]$ (since $x\in \alp_n$).

\medskip
\begin{remark}\label{vrr}
If $G$ is not a segment, the algebra $\alp_n$ can be more complicated, as the following example shows. Take $\alpha$ with $\alpha \not =5,4,3,\ldots ,\alpha^2-9\alpha-1\not =0$, $G=\{1,2,5\}$ and
$$
R_1(x)=x-1,\quad R_2(x)=x^2+1,\quad R_5(x)=x^5+x^4+x^3+1.
$$
Using Maple, one can check that $\Omega_G(n)=-12n^5+144n^4-628n^3+1296n^2-1280n+476\not =0$, $n\ge 0$. The polynomials $(q_n)_n$ satisfy then recurrence relations of the form (\ref{qrr}) for
$$
Q_0(x)=x^4-\frac{(\alpha-2)(\alpha-4)(\alpha+9)}{\alpha^2-9\alpha-1}x^3,\quad
Q_1(x)=x^5+\frac{5(\alpha-5)(\alpha-4)(\alpha-2)(\alpha-1)}{4(\alpha^2-9\alpha-1)}x^3.
$$
Computational evidence suggests that
$$
\alp_n=\{ p(x)x^6+c_0Q_0(x)+c_1Q_1(x)+c_2:p\in \RR[x],c_0,c_1,c_2\in \RR\}.
$$
\end{remark}

\section{The case $\alpha=1,\ldots, \max G$}

The case when $\alpha=1,\ldots, \max G$ is specially interesting because by chosing $\R_g$, $g\in G$, as in (\ref{rpm}) or (\ref{rpm1}), we get
Krall-Laguerre polynomials and hence orthogonality with respect to a measure (something impossible when $\alpha\not = 1,\ldots, \max G,$ as Corollary \ref{pjc} showed). In general, we can find orthogonality properties by modifying properly the bilinear form $\langle p,q\rangle$ defined by (\ref{innlag}). To do that we need to introduce some more auxiliary functions. We define the nonnegative integers
\begin{equation*}\label{lpm}
\xi_M=\max\{0,m-\alpha\}, \quad  \xi_m=-\min\{0,m-\alpha\}.
\end{equation*}
For $i=0,\ldots, m-1$ and $g\in G$, the function $U_{i,\xi}^g$ is defined by
\begin{align*}
U_{i,\xi}^g(x)&=(i-m+\alpha+1)_{\xi_M}\frac{-x^{-m+i}}{m}+\kappa_i^g\sum_{l=0}^g(\alpha-l)_l w_l^g x^{-l-1},
\end{align*}
with no assumptions on the numbers $\kappa_i^g$, $g\in G$.
When $g\ge \alpha$, we define
\begin{align*}
V_{i,\xi}^g(x)&=(i-m+\alpha+1)_{\xi_M}\frac{-x^{-m+i}}{m}+\kappa_i^g\sum_{l=0}^{\alpha-1}(\alpha-l)_l w_l^g x^{-l-1}.
\end{align*}
For $g\ge \alpha$, we have to change $\langle p,q\rangle_2$ in (\ref{innlag2}) by transforming a portion of the integral in a discrete Sobolev inner product, and when $1\le \alpha\le m-1$ we also have to change $\langle p,q\rangle_1$ in (\ref{ip1}) by transforming it into a (nonsymmetric) Sobolev inner product. More precisely, we define the bilinear form
\begin{align}\label{innlagxs}
\langle p,q\rangle_{\xi}& =\int_0^\infty p(x)q^{(\xi_M)}(x)\mu_{\xi_m}dx\\\nonumber
&\quad +\sum_{i=0}^{m-1}\frac{q^{(i)}(0)}{i!}\int_0^\infty p(x)\left(\sum_{g\in G;g<\alpha}U_{i,\xi}^g(x)+\sum_{g\in G;g\ge \alpha}V_{i,\xi}^g(x)\right)\mu_{\alpha}dx\\\nonumber
&\quad +\Gamma(\alpha)\sum_{i=0}^{m-1}\frac{q^{(i)}(0)}{i!}\sum_{g\in G;g\ge \alpha}\kappa_i^g\sum_{j=0}^{g-\alpha}\frac{p^{(j)}(0)}{j!}
\sum_{l=\alpha+j}^g(\alpha-l)_jw_l^g.
\end{align}

The key Lemma \ref{lsc} is still true.

\begin{lemma} \label{l5.2} For $k,u\ge 0$, with $u\ge k$, and $0\le i\le m-1$, we have
\begin{equation*}\label{fsc2}
\langle x^kL_u^\alpha (x),x^i\rangle_{\xi} =\Gamma(\alpha)\sum_{g\in G}\kappa _i^g \sum _{l=k}^g(\alpha-l)_kw_l^g\binom{u+l-k}{l-k}.
\end{equation*}
\end{lemma}

\begin{proof}
Proceeding as in the proof of Lemma \ref{lsc}, we get
\begin{align}\label{ipc}
\langle x^kL_u(x),x^i\rangle_\xi &=\Gamma(\alpha)\sum_{g\in G;g<\alpha}\kappa _i^g \sum _{l=k}^g(\alpha-l)_kw_l^g\binom{u+l-k}{l-k}\\\nonumber
&\quad +\Gamma(\alpha)\sum_{g\in G;g\ge \alpha}\kappa _i^g \sum _{l=k}^{\alpha-1}(\alpha-l)_kw_l^g\binom{u+l-k}{l-k}\\\nonumber
&\quad+\Gamma(\alpha)\sum_{g\in G;g\ge \alpha}\kappa_i^g\sum_{j=0}^{g-\alpha}\frac{(x^kL_u^\alpha(x))^{(j)}(0)}{j!}
\sum_{l=\alpha+j}^g(\alpha-l)_jw_l^g.
\end{align}
Using (\ref{dlp}), we have $\displaystyle (x^kL_u^\alpha (x))^{(j)}(0)=(-1)^{j-k}\binom{j}{k}k!\binom{\alpha+u}{\alpha+j-k}$.
Using now (\ref{fcc}), we get that the last part of the right-hand side of the previous formula is given by
\begin{align*}
\sum_{g\in G;g\ge \alpha}&\kappa_i^g\sum_{j=0}^{g-\alpha}\frac{(x^kL_u^\alpha(x))^{(j)}(0)}{j!}
\sum_{l=\alpha+j}^g(\alpha-l)_{j}w_l^g\\
&=\sum_{g\in G;g\ge \alpha}\kappa_i^g\sum_{j=0}^{g-\alpha}\frac{(-1)^{j-k}\binom{j}{k}k!\binom{\alpha+u}{\alpha+j-k}}{j!}\sum_{l=\alpha+j}^g(\alpha-l)_jw_l^g
\\
&=\sum_{g\in G;g\ge \alpha}\kappa_i^g\sum_{j=0}^{g-\alpha-k}\frac{(-1)^{j}\binom{\alpha+u}{\alpha+j}}{j!}\sum_{l=\alpha+j+k}^g(\alpha-l)_{j+k}w_l^g\\
&=\sum_{g\in G;g\ge \alpha}\kappa_i^g\sum_{l=\alpha+k}^g(\alpha-l)_kw_l^g
\sum_{j=0}^{l-\alpha-k}(-1)^{j}\binom{\alpha+j+k-l-1}{j}\binom{\alpha+u}{\alpha+j}\\
&=\sum_{g\in G;g\ge \alpha}\kappa_i^g\sum_{l=\alpha}^g(\alpha-l)_kw_l^g
\binom{u+l-k}{l-k}
\end{align*}
Substituting this in (\ref{ipc}) we get the result.
\end{proof}

\begin{remark}\label{remSobL}
Observe that the discrete part of the bilinear form \eqref{innlagxs} can not be represented in general as the discrete Laguerre-Sobolev bilinear form studied in \cite{ddI1}, which depends on an $m\times m$ matrix $M$. If we choose $\R_g$, $g\in G$, as in (\ref{rpm}) or (\ref{rpm1}), then we get Krall-Laguerre polynomials and hence, for some choice of the parameters involved, the bilinear form \eqref{innlagxs} reduces to the Krall-Laguerre measure \eqref{Kcwm}. 
We would like to stress here that there is a missprint in the matrix $M$ for the Krall-Laguerre case given at the end of \cite{ddI1}: the correct matrix $M=(M_{i,j})_{i,j=0}^{m-1}$ which gives the measure \eqref{Kcwm} is
$$
M_{i,j}=\begin{cases}
\displaystyle\binom{i+j}{i}b_{i+j},&i+j\leq m-1,\\
0,&i+j> m-1.
\end{cases}
$$
\end{remark}

We are now ready to prove the orthogonality conditions for $(q_n)_n$ (\ref{quschi}) with respect to the bilinear form (\ref{innlagxs}).
\begin{theorem}\label{l5.1} Let $\alpha $ be a positive integer  satisfying $ \alpha\le \max G $. Assume that conditions (\ref{mai}) hold and that the numbers $\kappa_i^g$, $i=0,\ldots, m-1$, $g\in G$, satisfy (\ref{epm1}) and (\ref{epm2}). Then, for $n\geq0$, the polynomials $(q_n)_n$ satisfy the following orthogonality properties with respect to the bilinear form (\ref{innlagxs}):
\begin{align*}
\langle q_n,q_i\rangle _{\xi}&=0,\quad i=0,\ldots , n-1,\\
\langle q_n,q_n\rangle _{\xi}&\not =0.
\end{align*}
With no assumptions on the real numbers $\kappa_j^g$, $j=0,\ldots, m-1$, $g\in G$, we always have
$$
\langle q_n(x),x^i\rangle_\xi=0, \quad n\ge m,\quad 0\le i\le n-1.
$$
\end{theorem}

\begin{proof}
Similar to the proofs of Theorem \ref{occ} and Corollary \ref{cor1}, using Lemma \ref{l5.2} instead of Lemma \ref{lsc}.
\end{proof}

The orthogonality conditions in the first part of Theorem \ref{l5.1} lead us to an improvement of Theorem \ref{t4.1} (the proof is similar  and it is omitted).

\begin{corollary}\label{t5.3} Let $\alpha$ be a positive integer satisfying $\alpha\le\max G$. Write $\rho=\max\{m,\max G-\alpha+1,\alpha\}$.
For any $p\in \RR[x]$ the sequence $(q_n)_n$ satisfies the recurrence relation
$$
x^{\rho}p(x)q_n(x)=\sum_{j=-s}^s\gamma_{n,j}q_{n+j}(x),\quad \gamma_{n,s}, \gamma_{n,-s}\not =0,
$$
where $s=\deg p+\rho$.
\end{corollary}

We finally prove one of the main results of this paper.

\begin{theorem}\label{t5.4} Let $\alpha$ be a positive integer with $\alpha\le \max G $, and assume
that conditions (\ref{mai}) hold. Then the sequence $(q_n)_n$ only satisfies a three-term recurrence relation as in (\ref{ttrr}) when the polynomials $\R_g(x)$ have the form (\ref{rpm}) or (\ref{rpm1}) (depending on whether $\alpha \ge m$ or $\alpha\le m-1$, respectively). Hence, for any other choice of $\R_g$, $g\in G$, the polynomials $(q_n)_n$ are no longer orthogonal with respect to a measure.
\end{theorem}

\begin{proof}
We proceed by \textit{reductio ad absurdum}.
Assume that the sequence $(q_n)_n$ satisfies the three-term recurrence relation
of the form (\ref{ttrr}).
Theorem \ref{l5.1} gives $0=\langle xq_n,1\rangle$, for $n\ge m+1$ (with no assumptions on the numbers $\kappa_i^g$). Using Lemma \ref{l5.2} (for $k=1$ and $i=0$) we can deduce that
$$
0=\langle xq_n,1\rangle=\Gamma(\alpha)\sum_{g\in G}\kappa_0^g\sum_{j=0}^m\beta_{n,j}\sum_{l=1}^g(\alpha-l)w_l^g\binom{n-j+l-1}{l-1}.
$$
For $\hat g$ in $G$, by taking $\kappa_0^g=\delta_{g,\hat g}$, we get
\begin{equation}\label{aet}
0=\sum_{j=0}^m\beta_{n,j}\sum_{l=1}^{\hat g}(\alpha-l)w_l^{\hat g}\binom{n-j+l-1}{l-1}.
\end{equation}
Write $r_{\hat g}$ for the polynomial
\begin{equation*}\label{aet2}
r_{\hat g}(x)=\sum_{l=1}^{\hat g}(\alpha-l)w_l^{\hat g}\binom{x+l-1}{l-1}.
\end{equation*}
Consider now the Casoratian determinant (\ref{defcd}) defined by the polynomials $r_{\hat g},\R_g, g\in G$: $\Phi _{\hat g}(x)=\Phi ^{r_{\hat g},\{\R_g\}_{g\in G}}(x)$.
From the definition of $\beta_{n,j}$ (\ref{defb}), the identity (\ref{aet}) can then be rewritten as $0=\Phi _{\hat g}(x)$. Consider now $\hat g=g_1=\min G$. If $\alpha\not =g_1$, the polynomial $r_{\hat g}$ has degree $g_1-1\not \in G$ ($w_{\hat g}^{\hat g}=1$), and so $0=\Phi _{\hat g}(x)$ is a contradiction because $\Phi _{\hat g}$ is a polynomial of degree $\hat g-1+\sum_{g\in G}g-\binom{m+1}{2}\ge 0$ (see (\ref{gcd})). Hence $\alpha=g_1$. Since we can always take a reduced representation of $r_{\hat g}$ (see definition in \eqref{aqj}), we have that $w_l^g=0$, $l=1,\ldots g_1-1$ and therefore $r_{\hat g}=0$. This gives for $\R_{g_1}$ the form
\begin{equation}\label{aet3}
\R_{g_1}(x)=\binom{x+\alpha}{\alpha}+a_0,
\end{equation}
for certain real number $a_0$.

Take now $\hat g=g_2$, the second element of $G$. Since $\alpha\not =g_2$, the polynomial $r_{\hat g}$ has degree $g_2-1$ and so $g_2-1\not \in G$, otherwise $0=\Phi _{\hat g}(x)$ is a contradiction because $\Phi _{\hat g}$ would be a polynomial of degree $g_2-1+\sum_{g\in G}g-\binom{m+1}{2}>0$ (see (\ref{gcd})). Hence, we conclude that $g_2-1 \in G$ and so $g_2=\alpha+1$. This gives (notice that $w_{g_2-1}^{g_2}=w_{g_1}^{g_2}=0$)
$$
r_{g_2}(x)=-\binom{x+\alpha}{\alpha}+\sum_{l=1}^{g_2-2}(\alpha-l)w_l^{g_2}\binom{x+l-1}{l-1}.
$$
Using (\ref{aet3}), we get $\Phi_{g_2}(x)=\Phi ^{\tilde r_{g_2},\{\R_g\}_{g\in G}}(x)$,
with
$$
\tilde r_{g_2}(x)=\sum_{l=1}^{g_2-2}(\alpha-l)w_l^{g_2}\binom{x+l-1}{l-1}+a_0.
$$
Proceeding as before, we can then conclude that $w_l^{g_2}=0$, $l=2,\ldots g_2-1$, and
$$
w_1^{g_2}=-\frac{a_0}{\alpha-1},
$$
(where we take $a_0=0$ and $w_1^{g_2}=0$ for $\alpha =1$). This gives
\begin{equation*}\label{aet4}
\R_{g_2}(x)=\binom{x+\alpha+1}{\alpha+1}-\frac{a_0}{\alpha-1}(x+1)+a_1,
\end{equation*}
for certain real number $a_1$.
We can now proceed in the same way to get that
$$
\R_{g_h}(x)=\begin{cases}
\displaystyle\binom{x+\alpha+h-1}{\alpha+h-1}+(h-1)!\sum_{l=0}^h\frac{(-1)^la_{h-l-1}}{(\alpha-l)_l}\binom{x+l}{l},&\alpha\ge m,\\
\displaystyle\binom{x+\alpha+h-1}{\alpha+h-1}+(h-1)!\sum_{l=0}^{h+\alpha-m-1}\frac{(-1)^la_{h-l-1}}{(\alpha-l)_l}\binom{x+l}{l},&\alpha\le m-1.
\end{cases}
$$
That is, $\R_g$ has the form (\ref{rpm}) or (\ref{rpm1}) (depending on whether $\alpha \ge m$ or $\alpha\le m-1$, respectively).
\end{proof}

For $\alpha=1,\ldots, \max G$, Theorem \ref{pzt} is no longer true, although the structure of $\alp_n$ can have a similar phenomenon to that displayed in Remark \ref{vrr}, as the following example shows. Take $\alpha=1$, $G=\{1,2,4\}$ and
$$
R_1(x)=x+2,\quad R_2(x)=x^2,\quad R_4(x)=x^4+1.
$$
We have that $\Omega_G(n)=-6n^4+16n^3+54n^2-208n+166\not =0$, $n\ge 0$. Using Maple one can see that the polynomials $(q_n)_n$ satisfy a recurrence relation of the form (\ref{qrr}) for
$$
Q(x)=x^3+\frac{6}{7}x^2.
$$
Computational evidence suggests that
$$
\alp_n=\{ p(x)x^4+c_0Q(x)+c_1:p\in \RR[x],c_0,c_1\in \RR\}.
$$

\section*{Disclosure statement}

No potential conflict of interest was reported by the authors.

\section*{Funding}

This work was partially supported by MTM2015-65888-C4-1-P (Ministerio de Econom\'ia y Competitividad), FQM-262 (Junta de Andaluc\'ia), Feder Funds (European Union), PAPIIT-DGAPA-UNAM grants IA102617 and IN104219 (M\'exico) and CONACYT grant A1-S-16202 (M\'exico).

\end{document}